\documentclass{amsart}
\usepackage{amssymb}
\usepackage[initials]{amsrefs}

\def\N{\mathbb N}
\def\R{\mathbb R}
\providecommand{\norm}[1]{\left\lVert#1\right\rVert}
\providecommand{\ip}[2]{\left\langle #1, #2 \right\rangle}
\newcommand{\F}{\mathrm F}

\allowdisplaybreaks
\numberwithin{equation}{section}

\theoremstyle{plain}
\newtheorem{theorem}{Theorem}[section]
\newtheorem{lemma}[theorem]{Lemma}
\newtheorem{corollary}[theorem]{Corollary}

\theoremstyle{remark}
\newtheorem{remark}[theorem]{Remark}

\title[Strongly quasinonexpansive mappings, II]
{Strongly quasinonexpansive mappings, II}

\author{Koji Aoyama}
\address[K.~Aoyama]{%
Department of Economics, Chiba University, 
Yayoi-cho, Inage-ku, Chi\-ba-shi, Chiba, 263-8522 Japan}
\email{aoyama@le.chiba-u.ac.jp}

\author{Kei Zembayashi}
\address[K.~Zembayashi]{%
Keio Girls Senior High School,
Mita, Minato-ku, Tokyo, 108-0073 Japan}
\email{zenbayashi@gshs.keio.ac.jp}


\begin{document}

\begin{abstract}
 This paper is devoted to the study of strongly quasinonexpansive
 mappings in an abstract  space and a Banach space. 
\end{abstract}

\maketitle

\section{Introduction}

One of the authors introduced the notion of strongly quasinonexpansive
mappings in a metric space in~\cite{pNACA2015}.
The notion is analogous to strong nonexpansiveness introduced in
Bruck and Reich~\cite{MR0470761}. 
Indeed, a strongly nonexpansive mapping in the sense of~\cite{MR0470761}
with a fixed point is strongly quasinonexpansive in the sense
of~\cite{pNACA2015} in the framework of Banach space. 

On the other hand, 
a mapping of type~(sr) was studied in~\cites{MR2884574, 
MR2529497, 
MR3258665} 
in a Banach space. 
Such a mapping is also analogous to a strongly nonexpansive mapping
in~\cite{MR0470761}; see also~Reich~\cite{MR1386686}.
However, 
a mapping of type~(sr) is different from 
a strongly quasinonexpansive mapping in the sense of~\cite{pNACA2015}. 

In this paper, in order to unify strong quasinonexpansiveness as above, 
we introduce and study a quasinonexpansive mapping, 
a strictly quasinonexpansive mapping,
and a strongly quasinonexpansive mapping in an abstract space.
In particular, we give some characterizations and
basic properties of such quasinonexpansive mappings.
Then, using these results, we obtain characterizations and properties
of mappings of type~(sr) in the sense of~\cites{MR2884574, 
MR2529497, 
MR3258665}. 

\section{Preliminaries}
Throughout this paper, 
$\R_+$ denotes the set of nonnegative real numbers, 
$\N$ the set of positive integers, 
$E$ a real Banach space, 
$E^*$ the dual of $E$, $\norm{\,\cdot\,}$ the norms of $E$ and $E^*$, 
$\ip{x}{x^*}$ the value of $x^*\in E^*$ at $x\in E$, 
and $J$ the duality mapping of $E$ into $2^{E^*}$, that is,
$Jx = \bigl\{ x^*\in E^* :
\ip{x}{x^*} = \norm{x}^2 = \norm{x^*}^2 \bigr \}$ for $x\in E$. 

It is known that 
the duality mapping $J$ is single-valued if $E$ is smooth;
$J$ is injective if $E$ is strictly convex,
that is, $Jx\cap Jy=\emptyset$ for all $x,y\in E$ with $x\ne y$;
see \cite{MR1864294} for more details. 
It is also known that every uniformly convex Banach space is
strictly convex. 

Let $X$ be a metric space with metric $d$ and $C$ a nonempty subset of
$X$.
A mapping $T\colon C \to X$ is said to be quasinonexpansive if $\F(T)$
is nonempty and $d(z,Tx) \leq d(z,x)$ for all $z \in \F(T)$ and
$x \in C$, where $\F(T)$ is the fixed point set of $T$, that is, 
$\F(T) = \{ z \in C: z = Tz\}$; 
$T$ is said to be strictly quasinonexpansive if
$\F(T)$ is nonempty and $d(z,Tx) < d(z,x)$ for all $z \in \F(T)$ and
$x \in C \setminus \F(T)$; 
$T$ is said to be strongly quasinonexpansive \cite{pNACA2015}
if $\F(T)$ is nonempty and for any $z \in F(T)$, $M > 0$, and 
$\epsilon > 0$ there exists $\delta > 0$ such that 
\[
 x \in C,\, d(z,x) \leq M,\, 
  d(z,x) - d(z,Tx) < \delta \Rightarrow d(Tx,x)< \epsilon.  
\]

\section{Lemmas}

In this section, we deal with two lemmas,
which play a central role in providing characterizations of 
strongly quasinonexpansive mappings in the next section. 

Throughout this section, we assume that $D$ is a nonempty set,
both $f$ and $g$ are bounded (above) functions of $D$ into $\R_+$,
and $\alpha = \sup\{f(x): x \in D\}$.

\begin{lemma}\label{lemma:gamma_well-defined}
 Suppose that $\alpha > 0$ and $\gamma(t)$ is defined by 
 \begin{equation} \label{eqn:def_gamma}
  \gamma (t) = \inf \{ g(x): x \in D,\, f(x) \geq t\}
 \end{equation}
 for $t \in [0, \alpha)$. Then the following hold: 
 \begin{enumerate}
  \item $\gamma$ is a nondecreasing bounded function of $[0,\alpha)$
	into $\R_+$; 
  \item if $x\in D$ and $f(x) \in [0, \alpha)$, then 
	$\gamma \bigl( f(x) \bigr) \leq g(x)$; 
  \item $\lim_{t \uparrow \alpha} \gamma(t)
	= \sup \{ \gamma (t): t \in [0,\alpha)\}< \infty$. 
 \end{enumerate}
\end{lemma}
\begin{proof}
 Let $t \in [0,\alpha)$ be fixed. 
 Then, by the definition of $\alpha$, 
 there exists $x\in D$ such that $f(x) > t$.  
 Thus $\{ g(x): x \in D,\, f(x) \geq t\}$ is nonempty, and hence
 $\gamma$ is a real-valued function defined on $[0,\alpha)$. 
 The conclusions (1) and (2) follow from the definition of $\gamma$;
 (3) follows from (1). 
 This completes the proof. 
\end{proof}

Using Lemma~\ref{lemma:gamma_well-defined}, we obtain the following: 

\begin{lemma}\label{lm:equiv}
 The following are equivalent: 
 \begin{enumerate}
  \item For any $\epsilon > 0$ there exists $\delta > 0$ such that
	\[
	 x \in D,\, g(x) < \delta \Rightarrow f(x) < \epsilon; 
	\]
  \item $f(x_n) \to 0$ whenever $\{x_n\}$ is a sequence in $D$ and
	$g(x_n) \to 0$; 
  \item there exists a nondecreasing bounded function $\gamma $ of
	$[0,\alpha]$ into $\R_+$ such that 
	$\gamma \bigl( f(x) \bigr) \leq g(x)$ for all $x \in D$ and 
	$\gamma(t)>0$ for all $t \in (0,\alpha]$. 
 \end{enumerate}
\end{lemma}
\begin{proof}
 We first show that (1) implies (3). 
 Suppose that $\alpha = 0$.
 In this case, the conclusion is clear by setting $\gamma(0)=0$.
 Next we suppose that $\alpha > 0$.
 Let $\gamma$ be a function of $[0,\alpha]$ into $\R_+$ 
 defined by~\eqref{eqn:def_gamma} for $t \in [0,\alpha)$ and
 $\gamma (\alpha) = \lim_{t \uparrow \alpha} \gamma(t)$. 
 Then it follows from Lemma~\ref{lemma:gamma_well-defined} that
 $\gamma$ is well defined and nondecreasing, and moreover,
 $\gamma \bigl( f(x) \bigr) \leq g(x)$ for all $x \in D$ with
 $f(x) \in [0,\alpha)$. 
 If $x \in D$ and $f(x) = \alpha$, then it is obvious that 
 $\gamma(t) \leq g(x)$ for all $t \in [0,\alpha)$ and hence
 \[
 \gamma(\alpha) = \lim_{t \uparrow \alpha} \gamma(t)
 = \sup \{\gamma (t): t \in [0, \alpha)\} \leq g(x). 
 \]
 Therefore, $\gamma \bigl( f(x) \bigr) \leq g(x)$ for all $x \in D$. 
 We finally show that $\gamma(t) > 0$ for all $t \in (0,\alpha]$. 
 Suppose that there exists $t \in (0,\alpha]$ such that $\gamma(t)=0$.
 Without loss of generality, we may assume that $t \ne \alpha$.
 By the definition of $\gamma$, there exists a sequence $\{y_n\}$ in $D$
 such that $f(y_n) \geq t$ and $g(y_n) < 1/n$ for all $n \in \N$. 
 On the other hand, by assumption, there exists $\delta > 0$ such that
 \[
  x \in D,\, g(x) < \delta \Rightarrow f(x) < t/2.
 \]
 Choosing $m \in \N$ with $1/m < \delta$,
 we have $t \leq f(y_m) \leq t/2$, which is a contradiction. 
 Therefore, $\gamma(t) > 0$ for all $t \in (0,\alpha]$. 

 We next show that (3) implies (2). 
 Let $\{x_n\}$ be a sequence in $D$ and suppose that $g(x_n) \to 0$
 and $f(x_n) \not\to 0$.
 Then there exist $\epsilon >0$ and a subsequence $\{x_{n_i}\}$ of
 $\{x_n\}$ such that $f(x_{n_i}) \geq \epsilon$ for all $i \in \N$. 
 Thus, by assumption, it follows that
 \[
 0 < \gamma (\epsilon) \leq \gamma \bigl( f(x_{n_i}) \bigr)
 \leq g(x_{n_i}) \to 0
 \]
 as $i \to \infty$, which is a contradiction.
 Therefore, $f(x_n) \to 0$. 

 We finally show that (2) implies (1). 
 Suppose that the conclusion does not hold.
 Then there exist $\epsilon > 0$ and a sequence $\{x_n\}$ in $D$ such
 that $g(x_n) < 1/n$ and $f(x_n) \geq \epsilon$ for all $n \in \N$. 
 Thus $g(x_n) \to 0$ and hence, by assumption, $f(x_n) \to 0$,
 which is a contradiction. 
\end{proof}

\section{Strongly quasinonexpansive mappings in an abstract space}

In this section, we introduce and study a quasinonexpansive mapping, 
a strictly quasinonexpansive mapping,
and a strongly quasinonexpansive mapping in an abstract space, 
which is a generalization of a metric space. 

Throughout this section, $X$ denotes a nonempty set, 
$\sigma$ a function of $X \times X$ into $\R_+$, 
and $\bar{B}(z, M)$ a subset of $X$ defined by
\[
 \bar{B}(z, M) = \{x \in X: \sigma(z,x) \leq M\}
\]
for $z \in X$ and $M > 0$. 

We say that the pair $(X,\sigma)$ satisfies the condition (S) if 
\[
 x \ne y \Leftrightarrow \sigma (x,y) > 0  
\]
for all $x,y \in X$; 
$(X,\sigma)$ satisfies the condition (B) if 
\begin{equation}\label{eq:condition(B)}
 \sup \{ \sigma(x,y): x,y \in \bar{B}(z,M) \} < \infty
\end{equation}
for all $z \in X$ and $M>0$; 
$(X,\sigma)$ satisfies the condition~(T) if for any $u\in X$, $M>0$, and
$\epsilon >0$ there exists $\eta >0$ such that
\begin{equation}\label{eq:Condition-T}
 x,y,z \in \bar{B}(u,M),\, \sigma(x,y) < \eta ,\, \sigma(y,z) < \eta
  \Rightarrow \sigma(x,z) < \epsilon. 
\end{equation}
It is clear that if $(X,\sigma)$ satisfies the condition~(S), then
$z \in \bar{B}(z,M)$, and hence $\bar{B}(z,M)$ is nonempty
for all $z \in X$ and $M>0$. 
It is also clear that $(X,\sigma)$ satisfies the condition~(T) if and
only if $\sigma(x_n,z_n) \to 0$ whenever 
$\{x_n\}$, $\{y_n\}$, and $\{z_n\}$ are sequences in $\bar{B}(u,M)$
for some $u \in X$ and $M>0$ such that
$\sigma(x_n,y_n) \to 0$ and $\sigma(y_n,z_n) \to 0$. 

\begin{remark}
 Suppose that $\sigma$ is a metric on $X$, that is, 
 $(X,\sigma)$ is a metric space.
 Then it is obvious that $(X,\sigma)$ satisfies the conditions~(S), (B),
 and~(T). 
\end{remark}

Let $C$ be a nonempty subset of $X$, $T$ a mapping of $C$ into $X$,
and $\F(T)$ the fixed point set of $T$. 
Inspired by~\cites{pNACA2015,MR667060,MR2884574,MR3213161,MR2529497}, 
we introduce the following: 
$T$ is said to be \emph{quasinonexpansive} (with respect to $\sigma$)
if $\F(T)$ is nonempty and
$\sigma(z, Tx) \leq \sigma(z,x)$ for all $z \in \F(T)$ and $x \in C$; 
$T$ is said to be \emph{strictly quasinonexpansive} (with respect to
$\sigma$) if $\F(T)$ is nonempty and 
$\sigma(z, Tx) < \sigma(z,x)$ for all $z \in \F(T)$ and $x \in C
\setminus \F(T)$; 
$T$ is said to be \emph{strongly quasinonexpansive} (with respect to
$\sigma$)
if $\F(T)$ is nonempty and for any $z \in \F(T)$, $M > 0$, and 
$\epsilon > 0$ there exists $\delta > 0$ such that 
\begin{equation}\label{e:def-QNSinX}
 x \in C\cap \bar{B}(z, M),\, 
  \sigma(z,x) - \sigma(z,Tx) < \delta \Rightarrow \sigma(Tx,x)< \epsilon. 
\end{equation}

It is clear from the definitions that
every strictly quasinonexpansive mapping is quasinonexpansive,
and moreover, we obtain the following: 

\begin{lemma}\label{lm:strong->stric}
 Let $C$ be a nonempty subset of $X$ and 
 $T$ a strongly quasinonexpansive mapping of $C$ into $X$. 
 Suppose that $(X,\sigma)$ satisfies the condition~(S).  
 Then $T$ is (strictly) quasinonexpansive.
\end{lemma}
\begin{proof}
 Suppose that $T$ is not strictly quasinonexpansive, that is, 
 there exist $z \in \F(T)$ and $y \in C \setminus \F(T)$
 such that $\sigma(z,Ty) \geq \sigma(z,y)$. 
 Then clearly $z \ne y$ and $Ty \ne y$. 
 Set $M = \sigma(z,y)$ and $\epsilon = \sigma(Ty,y)$.
 Taking into account the condition~(S),
 we see that $M >0$ and $\epsilon > 0$.
 Since $T$ is strongly quasinonexpansive, there exists $\delta > 0$ such
 that \eqref{e:def-QNSinX} holds. 
 Therefore, $\sigma(Ty, y) < \epsilon$, which is a contradiction. 
\end{proof}

In order to prove the next theorem, we need the following lemma: 

\begin{lemma}\label{lm:f,g:bdd}
 Let $C$ be a nonempty subset of $X$, $T$ a quasinonexpansive mapping
 of $C$ into $X$, $z\in \F(T)$, and $M> 0$. 
 Suppose that $(X,\sigma)$ satisfies the condition~(B). 
 Let $f$ and $g$ be functions defined by
 \begin{equation}\label{eq:def_f,g}
 f(x) = \sigma(Tx,x) \text { and }
 g(x) = \sigma(z,x) - \sigma(z,Tx)  
 \end{equation}
 for $x \in D$, where $D=C\cap \bar{B}(z,M)$.
 Then $f$ and $g$ are bounded functions of $D$ into $\R_+$. 
\end{lemma}
\begin{proof} 
 Since $T$ is quasinonexpansive, it follows that
 $0 \leq g(x) \leq \sigma(z,x) \leq M$ and $Tx \in \bar{B}(z,M)$ for all
 $x \in D$. Hence $g$ is a bounded function of $D$ into $\R_+$.
 Moreover, by the condition (B), we have 
 \[
 \sup \{ f(x): x \in D \} 
 \leq \sup \{ \sigma(x,y): x,y \in \bar{B}(z,M) \} < \infty. 
 \]
 Thus $f$ is a bounded function of $D$ into $\R_+$.
\end{proof}

We obtain the following characterization of strongly quasinonexpansive
mappings. 

\begin{theorem}\label{lm:SQNinX-equiv}
 Let $C$ be a nonempty subset of $X$ and $T$ a mapping of $C$ into $X$. 
 Suppose that $\F(T)$ is nonempty and $(X,\sigma)$ satisfies the
 conditions (S) and (B).
 Then the following are equivalent: 
 \begin{enumerate}
  \item $T$ is strongly quasinonexpansive; 
  \item $T$ is quasinonexpansive and $\sigma(Tx_n, x_n) \to 0$
	whenever $\{x_n\}$ is a sequence in $C \cap \bar{B}(z,M)$ 
	and $\sigma(z,x_n) - \sigma(z,Tx_n) \to 0$ for some $z \in \F(T)$
	and $M>0$. 
 \end{enumerate}
\end{theorem}

\begin{proof} 
 We first show that (1) implies (2). 
 Suppose that $T$ is strongly quasinonexpansive. Then
 Lemma~\ref{lm:strong->stric} implies that $T$ is quasinonexpansive.
 Moreover, suppose that $\{x_n\}$ is a sequence in $C \cap
 \bar{B}(z,M)$, and $\sigma(z,x_n) - \sigma(z,Tx_n) \to 0$ for some $z
 \in \F(T)$ and  $M>0$. 
 Set $D = C \cap \bar{B}(z,M)$.
 Let $f$ and $g$ be functions of $D$ into $\R_+$ defined
 by~\eqref{eq:def_f,g} for $x \in D$.
 Lemma~\ref{lm:f,g:bdd} shows that $f$ and $g$ are bounded functions of
 $D$ into $\R_+$.
 By the implication (1) $\Rightarrow$ (2) of Lemma~\ref{lm:equiv},
 we conclude that $\sigma (Tx_n,x_n) = f(x_n) \to 0$ and hence (2)
 holds. 

 We next show that (2) implies (1). Suppose that (2) holds. 
 Let $z \in \F(T)$, $M>0$, and $\epsilon >0$ be given. 
 Set $D = C \cap \bar{B}(z,M)$. 
 Let $f$ and $g$ be functions defined by~\eqref{eq:def_f,g} for
 $x \in D$. 
 Lemma~\ref{lm:f,g:bdd} shows that $f$ and $g$ are bounded functions of
 $D$ into $\R_+$.
 By the implication (2) $\Rightarrow$ (1) of Lemma~\ref{lm:equiv},
 we see that (1) holds.
\end{proof}

As a direct consequence of Theorem~\ref{lm:SQNinX-equiv},
we obtain the following: 

\begin{corollary}[{\cite{pNACA2015}*{Lemma 3.3}}]
 Let $X$ be a metric space with metric $d$, 
 $C$ a nonempty subset of $X$, and $T$ a mapping of $C$ into $X$. 
 Then the following are equivalent: 
 \begin{enumerate}
  \item $T$ is strongly quasinonexpansive; 
  \item $T$ is quasinonexpansive and $d(Tx_n, x_n) \to 0$
	whenever $\{x_n\}$ is a bounded sequence in $C$ 
	and $d(z,x_n) - d(z,Tx_n) \to 0$ for some $z \in \F(T)$. 
 \end{enumerate}
\end{corollary}

We provide another characterization of strongly quasinonexpansive
mappings as follows: 

\begin{theorem}\label{th:SQN-iff-gamma}
 Let $C$ be a nonempty subset of $X$ and $T$ a mapping of $C$ into $X$. 
 Suppose that $\F(T)$ is nonempty and $(X,\sigma)$ satisfies the
 conditions (S) and (B).
 Then the following are equivalent: 
 \begin{enumerate}
  \item $T$ is strongly quasinonexpansive; 
  \item for any $z \in \F(T)$ and $M>0$ there exists a nondecreasing
	bounded function $\gamma$ of $[0,\alpha]$ into $\R_+$ such that
	$\gamma(t)>0$ for all $t \in (0, \alpha]$ and
	\[
	\gamma \bigl( \sigma(Tx,x) \bigr) \leq \sigma(z,x) - \sigma(z,Tx)	 
	\]
	for all $x \in D$, where $D = C \cap \bar{B}(z,M)$ and $\alpha =
	\sup \{ \sigma (Tx,x) : x \in D\}$.
 \end{enumerate}
\end{theorem}
\begin{proof} 
 We first show that (1) implies (2). 
 Suppose that $T$ is strongly quasinonexpansive. As in the proof of
 Theorem~\ref{lm:SQNinX-equiv}, we see that $T$ is quasinonexpansive. 
 Let $z \in \F(T)$ and $M > 0$ be given. Set $D = C \cap \bar{B}(z, M)$.
 Let $f$ and $g$ be functions defined by \eqref{eq:def_f,g} for $x \in
 D$.
 Then Lemma~\ref{lm:f,g:bdd} shows that 
 $f$ and $g$ are bounded functions of $D$ into $\R_+$. 
 Using Lemma~\ref{lm:equiv}, we know that (2) holds.

 We next show that (2) implies (1). Suppose that (2) holds. 
 Then we see that $T$ is quasinonexpansive.
 Let $z \in \F(T)$ and $M > 0$ be given. 
 Let $f$ and $g$ be functions defined by \eqref{eq:def_f,g} for $x \in
 D = C \cap \bar{B}(z, M)$.
 Thus Lemmas~\ref{lm:f,g:bdd} and~\ref{lm:equiv},
 we conclude that (1) holds.
\end{proof}

As a direct consequence of Theorem~\ref{th:SQN-iff-gamma},
we obtain the following: 

\begin{corollary}\label{c:SQN-equiv-metric}
 Let $X$ be a metric space with metric $d$, 
 $C$ a nonempty subset of $X$, and $T$ a mapping of $C$ into $X$. 
 Suppose that $\F(T)$ is nonempty. 
 Then the following are equivalent:
 \begin{enumerate}
  \item $T$ is strongly quasinonexpansive; 
  \item for any $z \in \F(T)$ and $M>0$ there exists a nondecreasing
	bounded function $\gamma$ of $[0,\alpha]$ into $\R_+$ such that
	$\gamma(t)>0$ for all $t \in (0,\alpha]$ and
	\[
	\gamma\bigl( d(Tx,x) \bigr) \leq d(z,x) - d(z,Tx)	 
	\]
	for all $x \in C$ with $d(z,x) \leq M$,
	where $\alpha =	\sup \{ d(Tx,x) : x \in C,\, d(z,x) \leq M \}$.
 \end{enumerate} 
\end{corollary}
\begin{remark}
 Corollary~\ref{c:SQN-equiv-metric} is almost the same as 
 \cite{pNACA2015}*{Theorem 3.7}, which is a precise characterization
 of strongly quasinonexpansive mappings in a metric space. 
\end{remark}

We know that the class of strongly quasinonexpansive mappings in a
metric space is closed under composition~\cite{pNACA2015}*{Theorem 3.6}; 
see also~\cite{MR0470761}*{Proposition~1.1}. 
The class of strongly quasinonexpansive mappings with respect to
$\sigma$ has a similar property as follows: 

\begin{theorem}\label{l:FS-FT=FST}
 Let $C$ and $D$ be nonempty subsets of $X$, 
 and $S\colon C\to X$ and $T\colon D\to X$ quasinonexpansive mappings
 such that $T(D)\subset C$ and $\F(S) \cap \F(T)$ is nonempty. 
 Then the following hold: 
 \begin{enumerate}
  \item If $S$ or $T$ is strictly quasinonexpansive, then
	$\F(S) \cap \F(T) = \F(S T)$ and $ST$ is quasinonexpansive; 
  \item if both $S$ and $T$ are strongly quasinonexpansive
	and $(X,\sigma)$ satisfies the conditions~(S) and~(T), 
	then $ST$ is also strongly quasinonexpansive. 
 \end{enumerate}
\end{theorem}

\begin{proof}
 We first prove (1). 
 It is clear that $\F(S)\cap \F(T)\subset \F(ST)$ and hence $\F(ST)$ is
 nonempty. 
 We show that $\F(S)\cap \F(T)\supset \F(ST)$. 
 Let $z\in \F(ST)$ and $w\in \F(S)\cap \F(T)$ be given. 
 Since both $S$ and $T$ are quasinonexpansive, we have
 \[
 \sigma(w,z) = \sigma(w,STz) \leq \sigma(w,Tz) \leq \sigma(w,z). 
 \]
 This shows that 
 \begin{equation}
  \sigma(w,z)= \sigma(w,Tz) = \sigma(w,STz).  \label{eqn:1-1}
 \end{equation}
 Now suppose that $T$ is strictly quasinonexpansive. 
 Then $Tz = z$ from~\eqref{eqn:1-1}.  
 Hence $z=STz=Sz$. 
 Therefore we conclude that $z\in \F(S)\cap \F(T)$. 
 On the other hand, suppose that $S$ is strictly quasinonexpansive. 
 Then it follows from~\eqref{eqn:1-1}
 that $STz = Tz$. Thus $z=Tz$ and hence $z\in \F(S)\cap \F(T)$. 
 Consequently we know that $\F(S)\cap \F(T)=\F(ST)$. 
 This implies that $\sigma(z,STx)\leq \sigma(z,Tx)\leq \sigma(z,x)$ for all $x\in D$
 because $S$ and $T$ are quasinonexpansive. 
 Thus $ST$ is also quasinonexpansive. 

 We next show (2). 
 Lemma~\ref{lm:strong->stric} implies that $S$ and $T$ are strictly
 quasinonexpansive. Thus it follows from~(1) that
 $\F(ST)=\F(S)\cap \F(T)\ne \emptyset$. 
 Let $u \in \F(ST)$, $M>0$, and $\epsilon > 0$ be given. 
 By the condition~(T), 
 there exists $\eta >0$ such that \eqref{eq:Condition-T} holds. 
 Since $u \in \F(S)\cap \F(T)$ and both $S$ and $T$ are strongly
 quasinonexpansive, there exists $\delta > 0$ such that 
 \[
 x \in C \cap \bar{B}(u, M), \, \sigma(u,x) - \sigma(u,Sx) < \delta
 \Rightarrow \sigma(Sx,x) < \eta
 \]
 and
 \[
 x \in D \cap \bar{B}(u,M),\, \sigma(u,x) - \sigma(u,Tx) < \delta
 \Rightarrow \sigma(Tx,x) < \eta. 
 \]
 Suppose that $y \in D \cap \bar{B}(u, M)$ and 
 $\sigma(u,y) - \sigma(u, STy) < \delta$. 
 Since $S$ and $T$ are quasinonexpansive 
 and $u \in \F(S) \cap \F(T)$, we have
 \begin{multline*}
  \sigma(u,STy) \leq \sigma(u,Ty) \leq \sigma(u,y) \leq M,  \\
  \sigma(u,y) - \sigma(u,Ty) < \delta, \text{ and }
  \sigma(u,Ty) - \sigma(u,STy) < \delta. 
 \end{multline*}
 Thus $y, Ty, STy \in \bar{B}(u, M)$, $\sigma(STy,Ty) < \eta$, and
 $\sigma(Ty,y) < \eta$. 
 Therefore it follows from~\eqref{eq:Condition-T} that
 $(STy,y) < \epsilon$ and hence $ST$ is strongly quasinonexpansive.
\end{proof}

As a direct consequence of Theorem~\ref{l:FS-FT=FST}, we obtain the
following: 

\begin{corollary}[{\cite{pNACA2015}*{Lemma 3.5 and Theorem 3.6}}]
 Let $C$ and $D$ be nonempty subsets of a metric space $X$, 
 and $S\colon C\to X$ and $T\colon D\to X$ quasinonexpansive mappings
 such that $T(D)\subset C$ and $F(S) \cap F(T)$ is nonempty. 
 Then the following hold:
 \begin{enumerate}
  \item If $S$ or $T$ is strictly quasinonexpansive, then 
	$F(S) \cap F(T) = F(S T)$ and $ST$ is quasinonexpansive; 
  \item if $S$ and $T$ are strongly quasinonexpansive, 
	then $ST$ is also strongly quasinonexpansive. 
 \end{enumerate}
\end{corollary}

\section{Strongly quasinonexpansive mappings in a Banach space}
In this section, 
we give some characterizations of 
mappings of type~(sr) in the sense of
\cites{MR2884574, 
 MR2529497, 
MR3258665} 
and show a fundamental property of such mappings
by using the results in the previous section. 

Throughout this section, let $E$ be a smooth Banach space,
$\phi$ a function of $E\times E$ into $\R_+$ defined by 
\[
 \phi(x,y) = \norm{x}^2 - 2\ip{x}{Jy} + \norm{y}^2
\]
for $x,y\in E$, and $\bar B(z, M)$ a subset of $E$ defined by
\[
 \bar B (z, M) = \{x \in E: \phi(z,x) \leq M\}
\]
for $z \in E$ and $M>0$; see~\cite{MR1386667} for the function $\phi$. 

By definition, it is clear that
\begin{align}
 x=y  &\Rightarrow \phi(x,y) =0 \notag \\
 \intertext{and}
 ( \norm{x}-\norm{y})^2  &\leq \phi(x,y)
 \leq \left( \norm{x} + \norm{y}\right)^2 \label{eqn:phi-bounded}
\end{align}
for all $x,y\in E$. It is also clear that
if $E$ is strictly convex, then 
\begin{equation}\label{e:phi(x,y)=0:x=y}
\phi(x,y) = 0 \Rightarrow x=y  
\end{equation}
for all $x,y\in E$. Furthermore, we know the following: 

\begin{lemma}[\cite{MR1972223}]
 \label{lemma:kamimura-takahashi}
 Let $E$ be a smooth and uniformly convex Banach space
 and both $\{x_n\}$ and $\{y_n\}$ bounded sequences in $E$. 
 If $\phi(x_n,y_n)\to 0$, then $\norm{x_n - y_n}\to 0$. 
\end{lemma}

Under the appropriate assumptions,
we know that the pair $(E,\phi)$ satisfies the
conditions~(B), (S), and~(T) as stated in the previous section: 

\begin{lemma}\label{lm:E-(T)}
 Let $E$ be a smooth Banach space. The the following hold: 
 \begin{enumerate}
  \item $(E,\phi)$ satisfies the condition~(B); 
  \item if $E$ is strictly convex, then $(E,\phi)$ satisfies the
	condition~(S);
  \item if $E$ is uniformly convex, then $(E,\phi)$ satisfies the
	condition~(T).
 \end{enumerate}
\end{lemma}

\begin{proof}
 We first show (1). 
 Let $z \in E$ and $M > 0$.
 Then it is clear from~\eqref{eqn:phi-bounded} that
 \begin{align*}
  \sup \{ \phi(x,y): x,y \in \bar B(z,M)\}
  &\leq \sup \left\{ \left(\norm{x} + \norm{y} \right)^2 :
  x,y \in \bar B(z,M) \right\}\\
  &\leq \left( 2M + \norm{z} \right)^2 < \infty. 
 \end{align*}

 (2) immediately follows from~\eqref{e:phi(x,y)=0:x=y}. 
 
 Lastly, we show (3). 
 Let $\{x_n\}$, $\{y_n\}$, and $\{z_n\}$ be sequences in $\bar B(u, M)$
 for some $u \in E$ and $M>0$.
 Suppose that $\phi(x_n,y_n)\to 0$ and $\phi(y_n,z_n) \to 0$.
 Since $\{x_n\}$, $\{y_n\}$, and $\{z_n\}$ are bounded, 
 it follows from Lemma~\ref{lemma:kamimura-takahashi} that
 $\norm{x_n - y_n} \to 0$ and $\norm{y_n - z_n}\to 0$.
 Thus $\norm{x_n - z_n}\to 0$. Therefore we conclude that 
 \begin{align*}
  \phi(x_n,z_n) &= \norm{x_n}^2 - \norm{z_n}^2 - 2\ip{x_n-z_n}{Jz_n}\\
  &\leq (\norm{x_n} - \norm{z_n})(\norm{x_n} + \norm{z_n})
  + 2\norm{x_n-z_n}\norm{z_n} \to 0
 \end{align*}
 as $n\to 0$. This completes the proof. 
\end{proof}

Let $C$ be a nonempty subset of $E$ and $T$ a mapping of $C$ into $E$. 
Recall that a mapping $T$ is said to be \emph{quasinonexpansive with
respect to $\phi$}
if $\F(T) \ne \emptyset$ and $\phi(z,Tx)\leq \phi(z,x)$ for all
$z\in \F(T)$ and $x\in C$; 
$T$ is said to be \emph{strictly quasinonexpansive with respect to $\phi$}
if $\F(T) \ne \emptyset$ and $\phi(z,Tx) < \phi(z,x)$ for all
$z\in \F(T)$ and $x\in C \backslash \F(T)$; 
$T$ is said to be \emph{strongly quasinonexpansive with respect to $\phi$}
if $\F(T) \ne \emptyset$ and for any $z \in \F(T)$, $M>0$,
and $\epsilon> 0$ there exists $\delta >0$ such that
\[
 x \in C \cap \bar B(z,M),\, \phi(z,x) - \phi(z,Tx) <\delta
 \Rightarrow \phi(Tx,x) < \epsilon. 
\]

Using Theorems~\ref{lm:SQNinX-equiv}, \ref{th:SQN-iff-gamma},
and Lemma~\ref{lm:E-(T)}, we obtain the following characterizations of
strongly quasinonexpansive mappings with respect to $\phi$. 

\begin{theorem}\label{th:equiv-in-B}
 Let $C$ be a nonempty subset of a smooth and strictly convex Banach
 space $E$ and $T$ a mapping of $C$ into $E$.
 Suppose that $\F(T)$ is nonempty. Then the following are equivalent: 
 \begin{enumerate}
  \item $T$ is strongly quasinonexpansive with respect to $\phi$; 
  \item $T$ is quasinonexpansive with respect to $\phi$ and
	$\phi(Tx_n, x_n) \to 0$
	whenever $\{x_n\}$ is a sequence in $C \cap \bar{B}(z,M)$ 
	and $\phi(z,x_n) - \phi(z,Tx_n) \to 0$ for some $z \in \F(T)$
	and $M>0$; 
  \item for any $z \in \F(T)$ and $M>0$ there exists a nondecreasing
	bounded function $\gamma$ of $[0,\alpha]$ into $\R_+$ such that
	$\gamma(t)>0$ for all $t \in (0, \alpha]$ and
	\[
	\gamma \bigl( \phi(Tx,x) \bigr) \leq \phi(z,x) - \phi(z,Tx)	 
	\]
	for all $x \in D$, where $D = C \cap \bar{B}(z,M)$ and $\alpha =
	\sup \{ \phi (Tx,x) : x \in D\}$.
 \end{enumerate}
\end{theorem}

\begin{remark}
 A quasinonexpansive mapping with respect to $\phi$ is called a mapping
 of type~(r) in
 \cites{MR2884574, 
 MR2529497, 
 MR3013135, 
 MR2960628, 
 MR3258665}; 
 a mapping which satisfies the condition~(2) in
 Theorem~\ref{th:equiv-in-B} is called a mapping of type~(sr) in
 \cites{MR2884574, 
 MR2529497, 
MR3258665}. 
\end{remark}

Using Lemmas~\ref{lm:strong->stric}, \ref{lm:E-(T)}, and
Theorem~\ref{l:FS-FT=FST}, we obtain the following theorem, which is a
generalization of \cite{MR2884574}*{Lemma~3.2};
see also~\cite{MR2884574}*{Lemma 3.3}. 

\begin{theorem} 
 Let $C$ and $D$ be nonempty subsets of a smooth Banach space $E$, 
 and $S\colon C\to E$ and $T\colon D\to E$ quasinonexpansive mappings
 with respect to $\phi$ such that $T(D)\subset C$ and  $\F(S) \cap
 \F(T)$ is nonempty. Then following hold: 
 \begin{enumerate}
  \item If $S$ or $T$ is strictly quasinonexpansive with respect to
	$\phi$, then 
	$\F(S) \cap \F(T) = \F(S T)$ and $ST$ is quasinonexpansive
	with respect to $\phi$; 
  \item if $E$ is uniformly convex and 
	both $S$ and $T$ are  strongly quasinonexpansive with respect to
	$\phi$, then
	$ST$ is strongly quasinonexpansive with respect to $\phi$. 
 \end{enumerate}
\end{theorem}


\end{document}